\documentclass[12pt]{amsart}

\usepackage[english]{babel}
\usepackage{amssymb, amsmath, amsthm, upgreek, enumerate}
\usepackage[utf8]{inputenc}

\newtheorem{thm}{Theorem}
\newtheorem{lem}{Lemma}
\newtheorem{prop}{Proposition}
\newtheorem{cor}{Corollary}
\theoremstyle{definition}
\newtheorem{ex}{Example}

\AtBeginDocument{%
   \def\MR#1{}
}

\title[On the last nonzero digits of $n!$]{On the last nonzero digits of $n!$ in a given base}
\author{Bartosz Sobolewski}
\address{Jagiellonian University, Faculty of Mathematics and Computer Science, Institute of Mathematics, {\L}ojasiewicza 6, 30-348 Krak\'{o}w, Poland}
\email{\tt bartosz.sobolewski@doctoral.uj.edu.pl}

\begin{document}
\begin{abstract}
In this paper we study the sequence of strings of $k$ last nonzero digits of $n!$ in a given base $b$. We determine for which $b$ this sequence is automatic and show how to generate it using a uniform morphism. We also compute how often each possible string of $k$ digits appears as the $k$ last nonzero digits of $n!$.
\end{abstract}
\maketitle

\section{Introduction} \label{sec:Introduction}

The sequence $\{\ell_b (n!)\}_{n \geq 0}$ of the last nonzero digit in the base-$b$ expansion of $n!$ has been an object of interest of several authors. 
For example, Dresden \cite{Dr08} showed that the number
$$\sum_{n=0}^{\infty} \frac{\ell_{10}(n!)}{10^n}$$
is transcendental. 
Deshouillers and Ruzsa \cite{DR11} investigated the case $b=12$ and computed that for $a \in \{1,\ldots,11\}$ the asymptotic density of the set $$\{n \geq 0: \ell_{12}(n!) = a\}$$ is $1/2$ if $a \in \{4,8\}$  and $0$ otherwise. Recall that by the asymptotic density of a set $A \subset \mathbb{N}$ we mean the limit
$$\lim_{n \to \infty} \frac{ |A \cap \{1,\ldots,n\}|}{n},$$
if it exists. The sequence $\{\ell_{12}(n!)\}_{n \geq 0}$ was further studied by Deshouillers \cite{De12, De16}. Deshouillers and Luca \cite{DL10} proved that
$$ |\{ 0 \leq n \leq x: n! \text{~is~a~sum~of~three~squares} \}| = \frac{7}{8}x + O(x^{2/3}).$$
This can also be interpreted as a result concerning the last two nonzero digits of $n!$ in base $4$. Indeed, by Legendre's three-square theorem, $m$ is not a sum of three squares if and only if 
the string of two last nonzero digits of $m$ in base 4 is either $13$ or $33$.
%
All of these results rely on showing that the respective sequences are \emph{automatic}  or coincide with an automatic sequence on a set of asymptotic density 1, as in the case $b=12$.
We recall the notion of automaticity in Section \ref{sec:Automatic}. 

Motivated by these results, we will consider a general problem concerning the $k$ last nonzero digits of $n!$ in base $b$, where $k \geq 1$ is a fixed integer.  For a positive integer $m$ let $\ell_{b,k}(m)$ be the number whose base-$b$ expansion is given by $k$ last nonzero digits of $m$. We can add leading zeros if necessary. If $k=1$, then for simplicity we will use the previous notation $\ell_{b}(m)$. For a positive integer $c$ (not necessarily prime) denote by $\nu_c(m)$  the largest integer $l$ such that $c^l$ divides $m$, and by $m \bmod c$ the residue of $m$ modulo $c$ lying in the set $\{0,1,\ldots,c-1\}$ . Then we have
$$\ell_{b,k}(m) =\frac{m}{b^{\nu_b(m)}} \bmod b^k.$$
Obviously, $\ell_{b,k}(m)$ is not divisible by $b$ and satisfies $1 \leq \ell_{b,k}(m) \leq b^k-1$.

Our first aim is to study the automaticity of the sequence $\{\ell_{b,k}(n!)\}_{n \geq 0}$ and find (in the automatic case) its explicit description in terms of \emph{uniform morphisms}. Moreover, we would like to determine how often each $a \in \{1, \ldots, b^k-1\}$ appears in this sequence. More precisely, we are interested in computing the asymptotic density (if it exists) of the set 
$$ \{n \geq 0: \ell_{b,k}(n!) = a\},$$
which we will also call the \emph{asymptotic frequency} of $a$ in $\{\ell_{b,k}(n!)\}_{n \geq 0}$ and denote by $f_{b,k}(a)$.
We could also rephrase the latter question and ask how often a fixed string of digits $a_{k}\cdots a_1$, with $a_i \in \{0,\ldots,b-1\}$ and $a_1 \neq 0$, appears as $k$ last nonzero digits of $n!$.
 In Section \ref{sec:Main} we answer these questions in a general situation. 

\section{Main results}\label{sec:Main}
Consider the prime factorization of $b$:
$$b = p_1^{l_1} \cdots p_r^{l_r},$$
 where $p_1, \ldots, p_r$ are distinct primes and $l_1, \ldots, l_r$ are positive integers. 
If $r \geq 2$, then we can additionally renumber the primes so that the following technical conditions are satisfied:
 \begin{equation} \label{eq:primes_inequality_1}
 l_1(p_1-1) \geq l_2(p_2-1) \geq \cdots \geq l_r(p_r-1)
 \end{equation}
and 
\begin{equation} \label{eq:primes_inequality_2}
p_1 = \max\{p_i: l_i(p_i-1) = l_1(p_1-1)\}.
\end{equation}
It was recently proved by Lipka \cite{Li18} that $\{\ell_{b}(n!)\}_{n \geq 0}$ is automatic if and only if $b = p_1^{l_1}$ or $l_1(p_1-1) > l_2(p_2-1)$.
Our first result shows that a similar claim remains true for the sequence $\{\ell_{b,k}(n!)\}_{n \geq 0}$ describing the $k$ last nonzero digits. Moreover, we show that if $l_1(p_1-1) = l_2(p_2-1)$, then the sequence $\{\ell_{b,k}(n!)\}_{n\geq 0}$ is ``almost'' automatic. 
\begin{thm} \label{prop:automatic}
If either $b = p_1^{l_1}$ or $l_1(p_1-1) > l_2(p_2-1)$, then the sequence $\{ \ell_{b,k} (n!) \}_{n \geq 0}$ is $p_1$-automatic.
If $l_1(p_1-1) = l_2(p_2-1)$, then the sequence $\{ \ell_{b,k} (n!) \}_{n \geq 0}$  is not $m$-automatic for any $m$. However, it coincides with a $p_1$-automatic sequence on a set of asymptotic density 1.
\end{thm}
At the end of Section \ref{sec:proofs} we give a fairly effective way to generate $\{\ell_{b,k}(n!)\}_{n \geq 0}$  (or the corresponding $p_1$-automatic sequence) using uniform morphisms.

The following theorem gives for each number $a \in \{1,\ldots,b^k-1\}$ the frequency $f_{b,k}(a)$ of $a$  in $\{\ell_{b,k}(n!)\}_{n \geq 0}$. It turns out that these frequencies remain the same after restricting to subsequences along arithmetic progressions.

\begin{thm}\label{thm:density}
Let $a \in \{1, \ldots, b^k-1\}$ be such that $b {\nmid} a$. 
We have
$$f_{b,k}(a) =\frac{1}{(p_1-1)l_1} p_1^{\nu_{p_1}(a) -kl_1+1}$$
if $(b/p_1^{l_1})^k | a$, otherwise $f_{b,k}(a) = 0$.
Moreover, for any integers $c \geq 1$ and $d \geq 0$,  the number  $a$ appears in the subsequence $\{\ell_{b,k}((cn+d)!)\}_{n \geq 0}$ with asymptotic frequency $f_{b,k}(a)$, regardless of $c$ and $d$. 
\end{thm}
Note that for $b=12, k=1$ we obtain the result of Deshouillers and Ruzsa \cite{DR11}, whereas for $b=4, k=2$ we get a slightly weaker version of the result of Deshouillers and Luca \cite{DL10}.
We can easily deduce a corollary concerning a special case of $b$, for which all the nonzero frequencies are equal.
\begin{cor} \label{cor:special}
Assume that $l_1 = 1$ and let $a \in \{1,\ldots, b^k-1\}$ be such that $b {\nmid} a$. We have
 $$f_{b,k}(a) = \frac{1}{(p_1-1)p_1^{k-1}}$$
 if $(b/p_1)^k | a$, otherwise $f_{b,k}(a) = 0$
\end{cor}
In the following example we examine the last nonzero digit of $n!$ with $b =720$.

\begin{ex}
Consider the sequence $\{\ell_{720}(n!)\}_{n \geq 0}$ of the last nonzero digit of $n!$ in base $$b=720 = 5\cdot 3^2\cdot 2^4.$$ We label the primes $p_1 = 5, p_2 = 3, p_3 = 2$ and the exponents $l_1 = 1, l_2 = 2, l_3 = 4$ accordingly, so that conditions (\ref{eq:primes_inequality_1}) and (\ref{eq:primes_inequality_2}) are satisfied. In fact, we have the equality
$$l_1(p_1-1) = l_2(p_2-1) = l_3(p_3-1).$$
By Theorem \ref{thm:density} the sequence $\{\ell_{720}(n!)\}_{n \geq 0}$ is not automatic, however it coincides with a $5$-automatic sequence on a set of asymptotic density 1. A computation similar as in Proposition \ref{prop:smaller_alphabet} below shows that we can define such a 5-automatic sequence by $\beta(0) = 576$ and the recurrence relation
\begin{equation}\label{eq:example_recurrence}
\beta(5n+j) = \beta(n) \cdot j! \bmod{720}
\end{equation}
for all $n \geq 0$ and $j=0,1,2,3,4$. The terms $\beta(n)$ take values in the alphabet $\Sigma = \{144,288,432,576\}$. Define a $5$-uniform morphism $\varphi: \Sigma^{*} \rightarrow \Sigma^{*}$ by
$$\varphi(x) = x_0x_1x_2x_3x_4,$$
where $x_j =  x \cdot j! \bmod{720}$.
We can explicitly write
\begin{alignat*}{4}
\varphi(144) &= 144 \;144\;288\;144\;576, \qquad  \varphi(288) &= 288\;288\;576\;288\;432, \\  \varphi(432) &= 432\;432\;144\;432\;288, \qquad \varphi(576) &= 576\;576\;432\;576\;144.
\end{alignat*}
This morphism has exactly $4$ fixed points. By (\ref{eq:example_recurrence}) it is clear that $\{\beta(n)\}_{n \geq 0}$ is the fixed point of $\varphi$ starting with 576, so it is a \emph{pure morphic sequence} (does not require a coding). 

We have $\beta(n) = \ell_{720}(n!)$ on a set of asymptotic density 1, and hence the frequency of each symbol in both sequences is the same. In particular, if $a \not\in \Sigma$, then $f_{720}(a) = 0$. To compute the frequencies of $a \in \Sigma$, consider the $4 \times 4$ incidence matrix $M$ associated with $\varphi$, which is of the form:
$$
M = \begin{bmatrix}
 3 & 0 & 1 & 1 \\
 1 & 3 & 1 & 0 \\
 0 & 1 & 3 & 1 \\
 1 & 1 & 0 & 3 \\
\end{bmatrix},
$$
where the rows and colums correspond to the elements of $\Sigma$ arranged in ascending order. All the entries of $M^2$ are positive, and therefore $\varphi$ is primitive. The matrix $\frac{1}{5}M$ is row-stochastic, thus by the discussion in Section \ref{sec:Automatic} we have $f_{720}(a) = 1/4$ for $a \in \Sigma$.
\end{ex}

In general, bases $b$ such that $\{ \ell_{b,k} (n!) \}_{n \geq 0}$ is not automatic and the nonzero frequencies $f_{b,k}(a)$ are not all equal, require considerably larger alphabets $\Sigma$ (at least when using the approach presented in Section \ref{sec:proofs}).
One can check that the method of Proposition \ref{prop:smaller_alphabet} applied to $b = 144, k=1$ produces the "smallest" such example (in terms of $|\Sigma|$)  with $|\Sigma| = 48$.

\section{Preliminaries} \label{sec:Automatic}
We recall the definition of automatic sequences in terms of uniform morphisms. For a more detailed description see \cite[Chapters 4--6]{AS03}. Let $\Sigma$ and $\Delta$ be finite alphabets and denote by $\Sigma^*$ and $\Delta^*$ the sets of finite words created from letters in $\Sigma$ and $\Delta$, respectively, together with the \emph{empty word} $\epsilon$.
We call a map $\varphi : \Sigma^*  \rightarrow \Delta^*$ a \emph{morphism} if $\varphi(xy) = \varphi(x)\varphi(y)$ for all $x,y \in \Sigma^*$. Clearly, a morphism is uniquely determined by its values on single letters in $\Sigma$. We say that a morphism is \emph{$l$-uniform} for an integer $l \geq 1$ if $|\varphi(x)| = l$ for all $x \in \Sigma$, where $|y|$ denotes the length of a word $y$. A 1-uniform morphism is called a \emph{coding}.
If $\Sigma = \Delta$ then we denote by $\varphi^i$ the $i$-th iterate of $\varphi$  (with $\varphi^0$ being the identity morphism on $\Sigma^{*}$). A morphism $\varphi$ is said to be \emph{prolongable} on $x \in \Sigma$ if  $\varphi(x) = xy$ for some $y \in \Sigma^*$ and  $\varphi^{i}(y) \neq \epsilon$ for all $i \geq 0$.
If $\varphi$ is prolongable on $x$ then the sequence of words $x, \varphi(x), \varphi^2(x),\ldots$ converges to the infinite word  $$\varphi^{\omega}(x) = x\,y \,\varphi(y)\, \varphi^2(y) \cdots$$ in the sense that each $\varphi^i(x)$ is a prefix of $\varphi^{\omega}(x)$ for $i \geq 0$. We can naturally extend $\varphi$ to infinite words over $\Sigma$. Then one can check that $\varphi^{\omega}(x)$ is a \emph{fixed point} of $\varphi$, that is $\varphi(\varphi^{\omega}(x)) = \varphi^{\omega}(x)$. Moreover, it is the unique fixed point of $\varphi$ starting with $x$. A sequence $\{\alpha(n)\}_{n \geq 0}$ with values in $\Delta$, treated as an infinite word, is called a \emph{morphic sequence} if $\{\alpha(n)\}_{n \geq 0} = \uptau(\varphi^{\omega}(x))$ for some morphism $\varphi$ prolongable on $x$ and a coding $\uptau: \Sigma \rightarrow \Delta$. We call $\{\alpha(n)\}_{n \geq 0}$ an \emph{$l$-automatic sequence} if we can choose the morphism $\varphi$ to be $l$-uniform. To prove that a sequence $\{\beta(n)\}_{n \geq 0}$ is a fixed point of an $l$-uniform morphism $\varphi$ it is necessary and sufficient that for all $n \geq 0$ we have
$$\varphi(\beta(n)) = \beta(ln) \beta(ln+1) \cdots \beta((l+1)n-1).$$

The asymptotic  frequency of a letter in an automatic sequence does not always exist (see \cite[Example 8.1.2]{AS03}). To guarantee its existence it is enough to assume that the morphism $\varphi$ generating the sequence is \emph{primitive}, that is, there exists an integer $i \geq 1$ such that for all $x, y \in \Sigma$ the letter $y$ appears in $\varphi^i(x)$. Below we give an overview of a general method of finding the frequencies of symbols in $\varphi^{\omega}(x)$, where $\varphi$ is prolongable on $x \in \Sigma$. Once these frequencies are known, it is straightforward, given a coding $\uptau: \Sigma^{*} \rightarrow \Delta^{*}$, to compute the frequency of each $a \in \Delta$ in the infinite word $\uptau(\varphi^{\omega}(x))$. For more details and examples see \cite[Chapter 8]{AS03}. 

Let $\Sigma = \{a_1,\ldots,a_d\}$. For any word $w \in \Sigma^{*}$ denote by $|w|_{a_i}$ the number of ocurrences of $a_i$ in $w$. We associate with the morphism $\varphi$ the $d \times d$ \emph{incidence matrix} $M = [m_{i,j}]_{1 \leq i, j \leq d}$, 
where $m_{i,j} = |\varphi(a_j)|_{a_i}$. One can show that for any word $w \in \Sigma^{*}$ and integer $n \geq 1$ we have

$$\begin{bmatrix}
|\varphi^n(w)|_{a_1} \\
\vdots \\
|\varphi^n(w)|_{a_d}
\end{bmatrix} 
= M^n 
\begin{bmatrix}
|w|_{a_1} \\
\vdots \\
|w|_{a_d}
\end{bmatrix}.$$
The task of finding frequencies of letters in $\varphi^{\omega}(x)$ essentialy comes down to studying the behavior of $M^n$ as $n$ tends to infinity. A nonnegative square matrix $D$ is called \emph{primitive} if there exists an integer $n \geq 1$ such that $D^n$ has all entries positive. It is easy to see that $\varphi$ is primitive if and only if its incidence matrix $M$ is primitive. By \cite[Theorem 8.4.7]{AS03}, if $\varphi$ is a primitive morphism, then the frequencies of all letters in $\varphi^{\omega}(x)$ exist and are nonzero. Moreover, the vector of frequencies is the positive normalized eigenvector associated with the Perron--Frobenius eigenvalue of the incidence matrix $M$ (where the $i$-th entry of the vector corresponds to $a_i$). If additionally $M$ is a row-stochastic matrix multiplied by $C > 0$, then its Perron--Frobenius eigenvalue is equal to $C$ and the frequencies of $a_1, \ldots, a_d$ are all equal to $1/d$.

\section{Proofs}\label{sec:proofs}

Let us write the prime factorization of $b$:
$$b = p_1^{l_1} \cdots p_r^{l_r}$$
and denote $q_i = p_i^{l_i}$ for $i = 1,\ldots,r$. 
To begin with, in Lemmas \ref{lem:two_bases} and \ref{lem:dens_1} we show that if $b$ has at least two prime factors, then with the choice of $p_1$ satisfying (\ref{eq:primes_inequality_1}) and (\ref{eq:primes_inequality_2}),  the value $\ell_{b,k}(n!)$ depends only on $\ell_{q_1,k}(n!)$ and $\nu_{q_1}(n!)$ for almost all $n \geq 0$.

\begin{lem}\label{lem:two_bases}
Let $b$ have at least two distinct prime factors and let $m \geq 1$ be an integer. Assume that 
\begin{equation}\label{eq:val_ineq}
\nu_{q_1}(m) +k \leq \nu_{q_i}(m)
\end{equation}
 for all $i=2,\ldots,r$. Then
\begin{equation} \label{eq:two_bases}
\ell_{b,k}(m) = \ell_{q_1,k}(m)q^k t^{k+\nu_{q_1}(m) } \bmod b^k,
\end{equation}
where $q = b/q_1$ and $t$ is the multiplicative inverse of $q$ modulo $q_1^k.$
\end{lem}
\begin{proof}
We have
\begin{equation} \label{eq:base_b}
m = b^{\nu_b(m)}(b^kl + \ell_{b,k}(m)),
\end{equation}
where $l$ is some nonnegative integer. By the assumption (\ref{eq:val_ineq}) we have
$$0 \leq  \nu_{q_1}(b^kl + \ell_{b,k}(m)) \leq  \nu_{q_i}(b^kl + \ell_{b,k}(m)) - k,$$
for each $i = 2, \ldots, r$, which implies that $\nu_{q_i}(\ell_{b,k}(m))\geq k$. Therefore, we obtain
\begin{equation} \label{eq:chinese_1}
 \ell_{b,k}(m) \equiv 0 \pmod{q^k}.
 \end{equation}
  Now let 
\begin{equation} \label{eq:base_p}
m = q_1^{\nu_{q_1}(m)}(q_1^k j + \ell_{q_1,k}(m)),
\end{equation}
where $j$ is some nonnegative integer.
Again, by (\ref{eq:val_ineq}) we have $\nu_{b} (m) = \nu_{q_1}(m)$. Comparing (\ref{eq:base_b}) and (\ref{eq:base_p}) we get
$$q^{\nu_{q_1}(m)} \ell_{b,k}(m) \equiv \ell_{q_1,k}(m) \pmod{q_1^k},$$
or, equivalently,
\begin{equation} \label{eq:chinese_2}
\ell_{b,k}(m) \equiv t^{\nu_{q_1}(m)}  \ell_{q_1,k}(m) \pmod{q_1^k}.
\end{equation}
By applying the Chinese Remainder Theorem to the system of congruences (\ref{eq:chinese_1}) and (\ref{eq:chinese_2}), we finally obtain the result.
\end{proof}

Assume that the primes $p_1,\ldots,p_r$ are numbered in such a way that conditions (\ref{eq:primes_inequality_1}) and (\ref{eq:primes_inequality_2}) are satisfied.  Recall that for $p$ prime and $n \geq 0$ we have Legendre's formula:
$$\nu_p(n!) = \frac{n - s_p(n)}{p-1},$$
where $s_p(n)$ denotes the sum of digits of the base-$p$ expansion of $n$. The following lemma shows that for $m=n!$ the assumption of Lemma \ref{lem:two_bases} is satisfied on a set of $n$ of asymptotic density $1$.

\begin{lem} \label{lem:dens_1}
If $b$ has  at least two distinct prime factors, then the set 
$$
S_{b,k} = \left\{ n \geq 0: \nu_{q_i}(n!) \geq \nu_{q_1}(n!) + k  \text{ for } i=2,\ldots,r \right\}
$$
has asymptotic density 1. If furthermore $l_1(p_1 - 1) > l_2(p_2 - 1)$, then $S_{b,k}$ contains all but a finite number of nonnegative integers.
\end{lem}
\begin{proof}
 First, consider the case $l_1(p_1 - 1) > l_2(p_2 - 1)$. Fix any $2 \leq i \leq r$. By Legendre's formula we obtain
$$\lim_{n \to \infty} \frac{\nu_{q_i}(n!)}{\nu_{q_1}(n!)} = \frac{l_1(p_1 - 1)}{l_i(p_i - 1)} > 1,$$
and therefore all $n$ sufficiently large are in $S_{b,k}$.

Now consider the case $l_1(p_1-1) = l_i(p_i-1)$ for $i=2,\ldots,s$. By \cite[Lemma 3]{DR11}  there exists $\delta_i > 0$ such that the set of nonnegative integers $n$ satisfying
$$s_{p_i}(n) \leq s_{p_1}(n) - \delta_i \log n$$
has asymptotic density 1. For such $n$ we obtain
$$\nu_{q_i}(n!) = \left\lfloor\frac{n - s_{p_i}(n)}{l_i(p_i-1)} \right\rfloor \geq \left\lfloor\frac{n - s_{p_1}(n)+\delta_i \log n}{l_1(p_1-1)} \right\rfloor \geq \nu_{q_1}(n!) + \left\lfloor\frac{\delta_i \log n}{l_1(p_1-1)}\right\rfloor.$$
If we additionally take $n$ sufficiently large, then $n \in S_{b,k}$.
\end{proof}

When $b$ is a prime power we put $S_{b,k} = \mathbb{N}$. Then the formula (\ref{eq:two_bases}) is satisfied with $m = n!$ for any base $b \geq 2$ and all $n \in S_{b,k}$. We can further factorize the right side of (\ref{eq:two_bases}) into simpler terms.
 It is  easy to check that 
  $$\ell_{q_1,k}(m) = \ell_{p_1,l_1k}(m)p_1^{\nu_{p_1}(m) \bmod l_1} \bmod{q_1^k}.$$
If we additionally denote by $s$ the multiplicative order of $q$ modulo $q_1^k$, then the contribution of $\nu_{q_1}(m)$ to (\ref{eq:two_bases}) depends only on its residue modulo $s$. Using these observations, we can rewrite (\ref{eq:two_bases}) with $m = n!$ as
\begin{equation} \label{eq:two_bases_2}
\ell_{b,k}(n!) = \ell_{p_1,l_1k}(n!)p_1^{\nu_{p_1}(n!) \bmod l_1}q^k t^{k+ (\lfloor \nu_{p_1}(n!) / l_1 \rfloor \bmod{s})} \bmod b^k,
\end{equation}
whenever $n \in S_{b,k}$.

We will now consider the family of sequences 
$$
\{\alpha_{p,k,u,v}(n)\}_{n \geq 0}  = \{ (\ell_{p,k} (n!), \nu_p(n!) \bmod u, n \bmod v )\}_{n \geq 0},$$
where $p$ is a prime number, $k, u \geq 1$ are integers and $v$ is a positive integer divisible by $\mathrm{lcm} (p^{k-1},u,2)$. In the particular case $p=2, k=2$ we additionally require that $v$ is divisible by $4$. 
The motivation for this approach is that for $p=p_1,k'=l_1k,u=l_1s$ and $n \in S_{b,k}$ we can express $\ell_{b,k}(n!)$ by a coding of $\alpha_{p,k',u,v}(n)$, as seen in  (\ref{eq:two_bases_2}). The term $n \bmod v$ does not appear in the coding itself, however it plays a vital role in the recurrence relations describing the terms $\alpha_{p,k,u,v}(n)$. The advantage of considering general $u$ and $v$, instead of immediately focusing on $\{\ell_{b,k}(n!)\}_{n \geq 0}$, is that it allows to study this sequence along any arithmetic progression.

Below we briefly outline the structure of our reasoning.
We first prove in Proposition \ref{prop:morphism} that the sequence $\{\alpha_{p,k,u,v}(n)\}_{n \geq 0}$, treated as an infinite word, is a fixed point of a $p$-uniform morphism. Next, in Proposition \ref{prop:primitive} we show that this morphism is primitive. Proposition \ref{prop:equal_dens} uses the method described in Section \ref{sec:Automatic} to prove  that each of the possible values of $\alpha_{p,k,u,v}(n)$ appears with equal frequency. 
Combined with the relation between $\alpha_{p,k',u,v}(n)$ and $\ell_{b,k}(n!)$, this finally leads to the results of Section \ref{sec:Main}.

We start with deriving recurrence relations for $\alpha_{p,k,u,v}(n)$. For a fixed integer $c \geq 2$ we will use the notation
$$n_c! = 
\begin{cases}
\hspace{10pt} 1 &\text{ if } n=0, \\
\prod\limits_{\underset{\mathrm{gcd}(c,m) = 1}{1 \leq m \leq n}} \hspace{-2pt} m &\text{ if } n>0.
\end{cases}$$
This function is sometimes called the \emph{Gauss factorial} (see \cite{CD11} for some of its interesting properties). 
We begin with an auxiliary lemma.
\begin{lem} \label{lem:product}
We have
$$
(p^k)_{p}! \equiv
 \begin{cases}
 1  \pmod{p^k} & \text{ if } p=2 \text{ and } k \neq 2 , \\
 -1 \pmod{p^k} & \text{ otherwise. }
 \end{cases}
$$
%
%
\end{lem}
The proof is straightforward and relies on the fact that the product of all elements in a finite abelian group is equal to the product of the elements of order two.
The following lemma gives a recurrence relation satisfied by the terms $\ell_{p,k}(n!)$.
\begin{lem} \label{lem:recurrence}
For all $n \geq 0$ and $i = 0,\ldots,p-1$ we have
\begin{equation} \label{eq:recurrence}
\ell_{p,k}\left((pn+i)!\right) = \ell_{p,k}(n!) (pn+i)_{p}! \bmod p^k.
\end{equation}
\end{lem}
\begin{proof}
Obviously, (\ref{eq:recurrence}) is true for $n=0$ and $i=0$. By induction on $n$ we can compute that
\begin{align*}
\ell_{p,k}\left((p(n+1))!\right) &\equiv \ell_{p,k}\left((pn)!\right) \ell_{p,k}(pn+p) \prod_{i=1}^{p-1} \ell_{p,k}(pn+i) \\
&\equiv \ell_{p,k}(n!) (pn)_{p}! \,\ell_{p,k}(n+1) \prod_{i=1}^{p-1} (pn+i) \\
&\equiv \ell_{p,k}((n+1)!) (p(n+1))_{p}! \pmod{p^k}.
\end{align*}
This ends the proof for $i=0$. Since none of the numbers $pn+1, \ldots, pn+p-1$ is divisible by $p$, the result for $i > 0$ follows.
\end{proof}

We are now ready to show that the sequence $\{\alpha_{p,k,u,v}(n)\}_{n \geq 0}$ is a fixed point of a $p$-uniform morphism. The terms $\alpha_{p,k,u,v}(n)$ take values in the alphabet
$$\Lambda_{p,k,u,v} = (\mathbb{Z}/p^k \mathbb{Z})^{\times} \times(\mathbb{Z}/u \mathbb{Z})^{+} \times \{0,1,\ldots,v-1\},$$
where $(\mathbb{Z}/p^k \mathbb{Z})^{\times}$ and $(\mathbb{Z}/u \mathbb{Z})^{+}$ denote the multiplicative group modulo $p^k$ and the additive group modulo $u$, respectively. It is convenient to treat the first two coordinates of $(x,y,z) \in \Lambda_{p,k,u,v}$ as elements of a group, as will be seen in the proof of Lemma \ref{lem:group}. 
Define a $p$-uniform morphism $\psi_{p,k,u,v}: \Lambda_{p,k,u,v}^{*} \rightarrow \Lambda_{p,k,u,v}^{*}$ as follows:
$$\psi_{p,k,u,v}(x,y,z) = (x_0,y_0,z_0)(x_1,y_1,z_1)\cdots(x_{p-1},y_{p-1},z_{p-1}),$$ 
where 
 \begin{align*}
 x_i &= x (pz+i)_{p}! \bmod p^{k}, \\ 
 y_i &= y + z \bmod u, \\
 z_i &= pz+ i \bmod v.
 \end{align*} 
We have the following:
\begin{prop}\label{prop:morphism}
The sequence $\{\alpha_{p,k,u,v}(n)\}_{n \geq 0}$ is the fixed point of $\psi_{p,k,u,v}$ starting with $(1,0,0)$.
\end{prop}
\begin{proof}
Take any $n \geq 0$ and let
$$ (x,y,z) = (\ell_{p,k}(n!), \nu_p(n!) \bmod u, n \bmod v).$$
We can write $n = mv + z$ for some integer $m \geq 0$.
Fix any $0 \leq i \leq p-1$. By Lemmas \ref{lem:product} and \ref{lem:recurrence} we have
\begin{align*}
\ell_{p,k}((pn+i)!) &\equiv \ell_{p,k}(n!) (pn+i)_{p}! \equiv x (pvm + pz + i)_{p}! \\
&\equiv x ((p^k)_{p}!)^{mv/p^{k-1}}(pz+i)_{p}! \equiv x_i \pmod{p^{k}}.
\end{align*}
Note that in the case $p=2,k=2$ we have to use the assumption that $4 | v$  because $4_4! \equiv -1 \pmod{4}$.
Furthermore, we have
$$
\nu_p((pn+i)!) = \nu_p((pn)!) = \frac{pn - s_p(pn)}{p-1} =  \nu_p(n!) +n ,
$$
and therefore $\nu_p((pn+i)!) \equiv y_i \pmod{u}.$ Obviously, $pz + i \equiv z_i \pmod{v}$, which completes the proof.
\end{proof}
Observe that if $(x',y',z')$ appears on the $i$-th position in $\psi^{m}_{p,k,u,v}(x,y,z)$ for some integer $m \geq 1$, then $(x'x'',y'+y'',z')$ appears on the $i$-th position in $\psi^m_{p,k,u,v}(xx'',y + y'',z)$. We will use this property a number of times.
Our aim is now to prove that each symbol from $\Lambda_{p,k,u,v}$ appears in the sequence $\{\alpha_{p,k,u,v}(n)\}_{n \geq 0}$ with the same frequency. First, we give some auxiliary lemmas.
 
\begin{lem} \label{lem:equivalence}
For any two symbols $(x,y,z), (x',y',z') \in \Lambda_{p,k,u,v}$ denote 
$$
(x,y,z)R(x',y',z')
$$
if there exists an integer $m \geq 0$ such that $(x',y',z')$ appears in $\psi_{p,k,u,v}^m(x,y,z)$. Then $R$ is an equivalence relation on $\Lambda_{p,k,u,v}$.
\begin{proof}
The relation $R$ is obviously transitive. We will now prove that it is symmetric. Fix any $(x,y,z), (x',y',z')  \in \Lambda_{p,k,u,v}$ such that $(x,y,z)R(x',y',z')$.
For any integer $m \geq 0$ the third coordinate of $\psi_{p,k,u,v}^m(x',y',z')$ forms a finite sequence of $p^m$ consecutive integers taken modulo $v$, so   
we have $(x',y',z') R (xx'',y+y'',z)$ for some $(x'',y'') \in (\mathbb{Z}/p^k \mathbb{Z})^{\times} \times (\mathbb{Z}/u \mathbb{Z})^{+}$. If $(x'',y'') = (1,0)$, then we are done. Otherwise, transitivity gives  $(x,y,z)R(xx'',y+y'',z)$ and by using our earlier observation one can show inductively that $(xx'',y+y'',z)R(x(x'')^n,y+ny'',z)$  for all $n \geq 2$. If we choose $n$ to be the order of $(x'',y'')$ in $(\mathbb{Z}/p^k \mathbb{Z})^{\times} \times (\mathbb{Z}/u \mathbb{Z})^{+}$, then by transitivity we get $(x',y',z')R(x,y,z)$.
Hence, the relation $R$ is an equivalence relation, since each element of $\Lambda_{p,k,u,v}$ is in relation with at least one element.
\end{proof}
\end{lem}
In other words,  Lemma \ref{lem:equivalence} says that the automaton associated with $\psi_{p,k,u,v}$ is a disjoint union of its strongly connected components. To prove the primitivity of $\psi_{p,k,u,v}$, we need to show that all the elements of $\Lambda_{p,k,u,v}$ are related through $R$.
The following lemma displays a useful identity concerning the terms $\ell_{p,k}(n!)$.

\begin{lem} \label{lem:identity}
For any integers $n, m \geq 0$ not divisible by $p$, and integers $s \geq 0, t \geq k + \lfloor \log_p m \rfloor$ we have 
\begin{equation}\label{eq:identity}
\ell_{p,k}((p^{s+t}n - p^s m)!) \: \ell_{p,k}((p^s m)!) \:n \equiv  (-1)^{pm-1} \ell_{p,k}((p^{s+t}n)!)\:m \pmod{p^k}.
\end{equation}
\begin{proof}
We can write
\begin{equation} \label{eq:identity_congruence}
\ell_{p,k}((p^{s+t}n)!) \equiv  \ell_{p,k}((p^{s+t}n - p^s m)!) n \prod_{j=1}^{p^sm-1} \ell_{p,k}(np^{s+t}-j) \pmod{p^k}.
\end{equation}
The condition $t \geq k + \lfloor \log_p m \rfloor$ guarantees that for each $j = 1, \ldots, p^sm-1$ we have
$$ \ell_{p,k}(np^{s+t}-j) \equiv -\ell_{p,k}(j) \pmod{p^k}.$$
Hence, multiplying both sides of (\ref{eq:identity_congruence}) by $m$, we obtain the desired result.
\end{proof}
\end{lem}
In the following lemma we prove that $\alpha_{p,k,u,v}(n)$  takes on all the possible values.
\begin{lem} \label{lem:group}
Each symbol $(x,y,z) \in \Lambda_{p,k,u,v}$ appears in $\psi_{p,k,u,v}^{\omega}(1,0,0)$.
\begin{proof}
Consider the set 
$$H = \{(x,y) \in (\mathbb{Z}/p^{k} \mathbb{Z})^{\times} \times (\mathbb{Z}/u \mathbb{Z})^{+}: (x,y,0) \text{ appears in } \psi_{p,k,u,v}^{\omega}(1,0,0) \}.$$
By Proposition \ref{prop:morphism} we have $$H = \{ (\ell_{p,k} ((vn)!), \nu_p((vn)!) \bmod u): n \geq 0\}.$$
First, we will show that $H$ is a subgroup of $(\mathbb{Z}/p^{k} \mathbb{Z})^{\times} \times (\mathbb{Z}/u \mathbb{Z})^{+}$. Obviously, $(1,0) \in H$. Now if $(x,y), (x',y') \in H$ then $(1,0,0)R(x,y,0)$ and $(1,0,0)R(x',y',0)$. By a reasoning similar as in Lemma \ref{lem:equivalence} we obtain $(x,y,0) R (x^{-1}, -y,0)$, and therefore $(1,0,0)R(x^{-1}, -y,0)$. This means that $(x',y',0) R (x'x^{-1}, y'-y,0)$, thus $(x'x^{-1}, y'-y) \in H$. 

The next step is to show that in fact $H = (\mathbb{Z}/p^{k} \mathbb{Z})^{\times} \times (\mathbb{Z}/u \mathbb{Z})^{+}.$
To do this we will prove that a set of generators of $(\mathbb{Z}/p^{k} \mathbb{Z})^{\times} \times (\mathbb{Z}/u \mathbb{Z})^{+}$ appears in $H$. First, denote by $H_1$ the projection of $H$ on the first coordinate, which is a subgroup of $(\mathbb{Z}/p^{k} \mathbb{Z})^{\times}$.
Consider the congruence (\ref{eq:identity}) with $s = \nu_p(v), m = v/p^s$, $t \geq k+ \lfloor \log_p m \rfloor,$ and  let $n$ run  over the set $\{xm: x \in (\mathbb{Z}/p^{k} \mathbb{Z})^{\times}\}$. Then $pm$ is even and  (\ref{eq:identity}) takes the form
\begin{equation} \label{eq:congruence_x}
-\ell_{p,k}((p^txv - v)!) \: \ell_{p,k}(v!) x \equiv \ell_{p,k}((p^txv)!) \pmod{p^k}.
\end{equation}
 We have $\ell_{p,k}((p^txv - v)!), \ell_{p,k}(v!), \ell_{p,k}((p^txv)!) \in H_1$, and therefore $p^k-x \in H_1$ as well. It follows that $H_1 = (\mathbb{Z}/p^{k} \mathbb{Z})^{\times}$.
 
 Now fix $x=1$ in (\ref{eq:congruence_x}). By Lemmas \ref{lem:product} and \ref{lem:recurrence} we get
$$\ell_{p,k}((p^tv)!) \equiv \ell_{p,k}(v!) \prod_{i=1}^{t} (p^iv)_{p}! \equiv  \ell_{p,k}(v!) \pmod{p^k}.$$
Here we used the assumption that $v$ is divisible by $\mathrm{lcm}(p^{k-1},2)$.
The congruence (\ref{eq:congruence_x}) further simplifies to the form
$$\ell_{p,k}((v(p^t-1))!) \equiv -1 \pmod{p^k}$$ 
for any integer $t \geq k+ \lfloor \log_p m \rfloor$.
We have
\begin{align*}
\nu_p((v(p^{t+1}-1))!) - \nu_p((v(p^{t}-1))!) &= \\ \frac{p^tv(p-1) - s_p(m(p^{t+1}-1)) + s_p(m(p^{t}-1)) }{p-1} &=
p^tv - 1,
\end{align*}
since for $t \geq 1 + \lfloor \log_p m \rfloor$ the base-$p$ expansion of $mp^{t+1}-m$ has exactly one more digit (equal to $p-1$) than  $mp^{t}-m$, with other digits unchanged.  
If we denote $y = \nu_p((v(p^{t}-1))!)$, then it follows that $(p^k-1, y), (p^k-1,y-1) \in H$,
and hence $(1,1) \in H$. Combined with the earlier reasoning, this means that $H = (\mathbb{Z}/p^{k} \mathbb{Z})^{\times} \times (\mathbb{Z}/u \mathbb{Z})^{+}$, as $H$ contains a set of generators of the latter group.

Now take any $(x,y,z) \in \Lambda_{p,k,u,v}$. Directly from the definition of the morphism $\psi_{p,k,u,v}$, there exist some $(x',y') \in (\mathbb{Z}/p^{k} \mathbb{Z})^{\times} \times (\mathbb{Z}/u \mathbb{Z})^{+}$ such that $(x',y',0)R(x,y,z)$ and the result follows.
\end{proof}
\end{lem}

We are now ready to prove primitivity of $\psi_{p,k,u,v}$.

\begin{prop}\label{prop:primitive}
The morphism $\psi_{p,k,u,v}$ is primitive.
\end{prop}
\begin{proof}
By Lemma \ref{lem:group} we have $(1,0,0) R (x,y,z)$ for each $(x,y,z) \in \Lambda_{p,k,u,v}$, and thus all the elements of $\Lambda_{p,k,u,v}$ are related. In consequence, the incidence matrix $M = [m_{i,j}]_{1 \leq i,j \leq d}$ associated with $\psi_{p,k,u,v}$ is \emph{irreducible}, i.e., for each $i,j$ there exists a positive integer $n$ such that  $m^{(n)}_{i,j} > 0$, where $M^n = [m^{(n)}_{i,j}]_{1 \leq i,j \leq d}$. However, $M$ has a nonzero diagonal element corresponding to $(1,0,0) \in \Lambda_{p,k,u,v}$, so by a well-known fact it is primitive.
\end{proof}
 
Proposition \ref{prop:primitive} implies that the frequencies of all symbols $(x,y,z) \in \Lambda_{p,k,u,v}$ in the sequence $\{\alpha_{p,k,u,v}(n)\}_{n \geq 0}$ are positive. In the following proposition we show that they are in fact all equal.
 
\begin{prop}\label{prop:equal_dens}
Each element $(x,y,z) \in \Lambda_{p,k,u,v}$ appears in $\{\alpha_{p,k,u,v}(n)\}_{n \geq 0}$ with frequency
$$\frac{1}{p^{k-1}(p-1)uv}.$$
Equivalently, for any integers $c \geq 1$ and $0 \leq d \leq c-1$ each element $(x,y) \in (\mathbb{Z}/p^{k} \mathbb{Z})^{\times} \times (\mathbb{Z}/u \mathbb{Z})^{+}$  appears in  $\{\ell_{p,k}((cn+d)!), \nu_p((cn+d)!) \bmod u\}_{n \geq 0}$  with frequency
$$\frac{1}{p^{k-1}(p-1)u}.$$
\end{prop}
\begin{proof}
By the discussion in Section \ref{sec:Automatic} and Proposition \ref{prop:primitive} it suffices to prove that the matrix $\frac{1}{p} M$ is row-stochastic. In other words, we need to show that each fixed $(x',y',z') \in \Lambda_{p,k,u,v}$ appears exactly $p$ times in the words $\psi_{p,k,u,v}(x,y,z)$ with $(x,y,z) \in \Lambda_{p,k,u,v}$.

If $(x',y',z')$ appears in $\psi_{p,k,u,v}(x,y,z)$, then $(x' x'',y' + y'',z')$ appears the same number of times in $\psi_{p,k,u,v}(xx'',y + y'',z)$ for any $(x'',y'') \in (\mathbb{Z}/p^{k} \mathbb{Z})^{\times} \times (\mathbb{Z}/u \mathbb{Z})^{+}$. Therefore, for each fixed $z'$ the number of occurrences of $(x',y',z')$ in the words $\psi_{p,k,u,v}(x,y,z)$ with $(x,y,z) \in \Lambda_{p,k,u,v}$ does not depend on $x'$ and $y'$. 

Now we will show that it does not depend on $z'$ either.
The symbol $(x',y',z')$ appears on the $i$-th position in $\psi_{p,k,u,v}(x,y,z)$ for some $(x,y,z) \in \Lambda_{p,k,u,v}$ if and only if the congruence 
\begin{equation}\label{eq:stochastic_cong}
pz + i \equiv z' \pmod{v}.
\end{equation}
is satisfied.
Regardless of $z'$, the congruence (\ref{eq:stochastic_cong}) has exactly $p$ solutions $(z,i) \in \{0,1,\ldots,v-1\} \times\{0,1,\ldots,p-1\}$,
which completes the proof. 
\end{proof}
According to Proposition \ref{prop:equal_dens}, we could informally say that for a random integer $n \geq 0$ the three coordinates of $\alpha_{p,k,u,v}(n)$ behave like independent uniformly distributed random variables.

We will now proceed to prove the results of Section \ref{sec:Main}. As in the beginning of this section, write 
$$b = p_1^{l_1} \cdots p_r^{l_r}$$
and assume that the conditions (\ref{eq:primes_inequality_1}) and (\ref{eq:primes_inequality_2}) are satisfied.
Until the end of this section let $q_1 = p_1^{l_1}, q = b/q_1$ and denote by $t$ and $s$ the multiplicative inverse and the order of $q$ modulo $q_1^k,$ respectively. We assign specific values to the parameters $p,k',u,v$ of the sequence $\{\alpha_{p,k',u,v}(n)\}_{n \geq 0}$. Put $u = l_1s$ and
$$v = \begin{cases}
\mathrm{lcm}(p_1^{kl_1-1},2,u) &\text{ if } b^k \neq 4, \\
 4 &\text{ if either } b=2, k=2 \text{ or }b=4, k=1.
\end{cases}.$$
To appropriately describe $\{\ell_{b,k}(n!)\}_{n \geq 0}$, define
\begin{align*}
\Sigma_{b,k} &= \Lambda_{p_1,l_1k,u,v}, \\
\varphi_{b,k} &= \psi_{p_1,l_1k,u,v}.
\end{align*}
The terms $\ell_{b,k}(n!)$ take values in the alphabet 
$$\Delta_{b,k} = \{ 1 \leq a \leq b^k: b \nmid a\}.$$
We also define a coding $\uptau_{b,k} : \Sigma_{b,k}^{*} \rightarrow \Delta_{b,k} ^{*}$ in the following way:
$$\uptau_{b,k} (x,y,z) = x p_1^{(y \bmod l_1)}q^k t^{k+(\lfloor y/l_1 \rfloor \bmod s)}  \bmod b^k$$
and denote $$\{\beta_{b,k}(n)\}_{n \geq 0} = \uptau_{b,k}(\varphi_{b,k}(1,0,0)).$$
The terms $\beta_{b,k}(n)$ take values in the alphabet 
$$\widetilde{\Delta}_{b,k} = \{ a \in \Delta_{b,k}: q^k | a \}.$$
The formula (\ref{eq:two_bases_2}) shows that $\beta_{b,k}(n) = \ell_{b,k}(n!)$ for $n \in S_{b,k}$. 
We are now ready to prove Theorem \ref{prop:automatic}.

\begin{proof} [Proof of Theorem \ref{prop:automatic}]
The sequence $\{\beta_{b,k}(n)\}_{n \geq 0}$ is $p_1$-automatic, which immediately gives the desired result when $b$ is a prime power.  If $b$ has multiple prime factors and $l_1(p_1-1) > l_2(p_2-1)$, then Lemma \ref{lem:dens_1} guarantees that $\beta_{b,k}(n) = \ell_{b,k}(n!)$ for all but a finite number of nonnegative integers. However, an $l$-automatic sequence remains $l$-automatic after changing a finite number of terms. Non-automaticity of $\{ \ell_{b,k} (n!) \}_{n \geq 0}$ in the case $l_1(p_1-1) = l_2(p_2-1)$ follows immediately from the result of Lipka \cite{Li18}, since $\ell_{b} (n!)$ is a coding of $\ell_{b,k} (n!)$.
\end{proof}

In the case when $\{ \ell_{b,k} (n!) \}_{n \geq 0}$ is automatic, we use the description of $\{\beta_{b,k}(n)\}_{n \geq 0}$  in order to express the former sequence as the image under a coding of a fixed point of a uniform morphism. Indeed, we can add finitely many extra letters to $\Sigma_{b,k}$  to handle the terms $\ell_{b,k}(n!) \neq \beta_{b,k}(n)$ and modify $\varphi_{b,k}$ and $\uptau_{b,k}$ accordingly. 
In the following proof of Theorem \ref{thm:density} we rely on directly studying $\beta_{b,k}(n)$ instead of $\ell_{b,k}(n!)$.

\begin{proof} [Proof of Theorem \ref{thm:density}]
The frequency of each $a \in \Delta_{b,k}$ is the same in $\{\ell_{b,k}(n!)\}_{n \geq 0}$ and $\{\beta_{b,k}(n)\}_{n \geq 0}$, as these sequences agree on a set of $n$ of density 1, thus we can focus on the latter sequence.
Obviously, if $a \not\in \widetilde{\Delta}_{b,k}$, then this frequency equals $0$.
 
The value of $\uptau_{b,k}(x,y,z)$ does not depend on $z$, so we can define an alphabet $\Sigma'_{b,k}$ as the first two components of $\Sigma_{b,k}$ and a coding  $\uptau'_{b,k}(x,y): (\Sigma'_{b,k})^* \rightarrow \Delta_{b,k}^{*}$ by  $\uptau'_{b,k}(x,y) = \uptau_{b,k}(x,y,z)$.
In the case when $q^k | a$, the equality $\uptau'_{b,k}(x,y) = a$ is equivalent to the system of congruences
\begin{align*}
y &\equiv \nu_{p_1}(a) \pmod{l_1}, \\
x  t^{\lfloor y/l_1 \rfloor}  p_1^{\nu_{p_1}(a)}&\equiv a  \pmod{q_1^{k}},
\end{align*}
which has exactly $s p_1^{\nu_{p_1}(a)}$ solutions $(x,y) \in  \Sigma'_{b,k}$. The alphabet $\Sigma'_{b,k}$ has exactly $p_1^{l_1k-1}(p_1-1)l_1s$ elements, therefore by Proposition \ref{prop:equal_dens} the symbol $a$ appears in $\{\beta_{b,k}(n)\}_{n \geq 0}$ and its subsequences $\{\beta_{b,k}(cn+d)\}_{n \geq 0}$ with asymptotic frequency
$$f_{b,k}(a) = \frac{1}{(p_1-1)l_1} p_1^{\nu_{p_1}(a) -kl_1+1},$$ 
which ends the proof.
\end{proof}

It is in fact possible to generate the sequence $\{\beta_{b,k}(n)\}_{n \geq 0}$  using  an alphabet smaller (in most cases) than $\Sigma_{b,k}$. Let
$$\widetilde{\Sigma}_{b,k} = \widetilde{\Delta}_{b,k}  \times \{0,\ldots,v-1\}$$
and define a $p_1$-uniform morphism $\widetilde{\varphi}_{b,k} : \widetilde{\Sigma}_{b,k}^* \rightarrow \widetilde{\Sigma}_{b,k}^*$ by 
$$\widetilde{\varphi}_{b,k}(x,z) = (x_0,z_0)(x_1,z_1)\cdots(x_{p-1},z_{p-1}),$$ 
where for $j=0,1,\ldots,p_1-1$, we put
 \begin{align*}
 x_j &= x p_1^{z \bmod l_1} t^{(\lfloor z/l_1 \rfloor +\varepsilon(x,z)) \bmod{s}} (p_1z+j)_{p_1}! \bmod b^{k}, \\ 
 z_j &= p_1z+ j \bmod v,
 \end{align*} 
 and
 $$
\varepsilon(x,z) = \begin{cases} 
0 & \text{ if } \nu_{p_1}(x) \bmod{l_1} + z \bmod{l_1} < l_1, \\
1 & \text{ otherwise.} \end{cases}
$$
In particular, if $l_1 = 1$, then $x_j$ takes a simpler form
$$x_j = x  t^{z \bmod{s}} (-1)^z j! \bmod{b^k}.$$
Let $\widetilde{\uptau}_{b,k}: \widetilde{\Sigma}_{b,k}^* \rightarrow \widetilde{\Delta}_{b,k}^*$ be a coding defined by $\widetilde{\uptau}_{b,k}(x,y) = x$. The following proposition shows that the sequence $\{\beta_{b,k}(n)\}_{n \geq 0}$ can also be described in terms of $\widetilde{\varphi}_{b,k}$ and $\widetilde{\uptau}_{b,k}$.

\begin{prop}\label{prop:smaller_alphabet}
 We have $\{\beta_{b,k}(n)\}_{n \geq 0} = \widetilde{\uptau}_{b,k}(\widetilde{\varphi}_{b,k}^{\omega}(q^kt^k,0)).$ 
\end{prop}
\begin{proof}
It suffices to show that $$\widetilde{\varphi}_{b,k}^{\omega}(q^kt^k,0) = \{(\beta_{b,k}(n), n \bmod v)\}_{n \geq 0}.$$  By definition $\beta_{b,k}(0) = q^kt^k$. Let $n \geq 0$ and denote $(x,z) = (\beta_{b,k}(n), n \bmod v)$. Using the recurrence relations in Lemma \ref{lem:recurrence} and  Proposition \ref{prop:morphism}, we get for each for $j=0,1,\ldots,p_1-1$:
\begin{align*}
\beta_{b,k}(p_1n+j) &\equiv \ell_{p_1,l_1k}((p_1n+j)!)p_1^{\nu_{p_1}((p_1n+j)!) \bmod l_1}q^k t^{k+ (\lfloor \nu_{p_1}((p_1n+j)!) / l_1 \rfloor \bmod{s})} \\
&\equiv \ell_{p_1,l_1k}(n!) (p_1n+j)_{p_1}! \: p_1^{(\nu_{p_1}(n!)+n) \bmod l_1} q^k t^{k+ (\lfloor (\nu_{p_1}(n!)+n) / l_1 \rfloor \bmod{s})} \nonumber \\
&\equiv \beta_{b,k}(n) (p_1z+j)_p! \: p_1^{z \bmod l_1} t^{(\lfloor z/l_1 \rfloor +\varepsilon(x,z)) \bmod{s}}  \nonumber \\
&\equiv x_j \pmod{b^k}. \nonumber
\end{align*}
The result follows immediately.
\end{proof}
Observe that 
$$|\widetilde{\Sigma}_{b,k}| = q_1^{k-1}(q_1-1)v \leq q_1^{k-1}(q_1-q_1/p_1)uv = |\Sigma_{b,k}| $$
with equality if and only if $u=1$ or, equivalently, $l_1 = s = 1$. In other words, $\widetilde{\Sigma}_{b,k}$ contains roughly $u$ times less elements than $\Sigma_{b,k}$. 

\section{Acknowledgements} 
I would like to thank Jakub Byszewski for careful proofreading of the manuscript and many helpful suggestions.

\bibliographystyle{amsplain}
\bibliography{references}
\end{document}